\documentclass[11pt]{amsart}
\usepackage[margin=1in]{geometry}

\usepackage[utf8]{inputenc}

\usepackage{graphbox} 
\usepackage{mathtools, amsthm, amsfonts, amssymb, amsmath}
\usepackage{accents} 
\usepackage{microtype}
\usepackage{enumerate}
\usepackage{stackengine}

\usepackage{multirow}
\usepackage{array}
\usepackage{longtable}
\usepackage{thmtools}
\usepackage[margin=1in]{geometry}
\usepackage{xcolor}
\usepackage{tikz}
\usepackage{tikz-cd}
\usepackage{tikz-3dplot}
\usetikzlibrary{patterns,external}
\usetikzlibrary{calc}
\usetikzlibrary{shapes.multipart}

\usepackage[numbers,sort]{natbib}

\newcommand{\rect}{\raisebox{-1.5pt}{\hspace{-4pt}\fontsize{90}{90} \textbf{-}}}

\definecolor{fondo}{rgb}{0.898,0.996,0.898}
\definecolor{diagrammagenta}{RGB}{255,0,255}
\pgfkeys{/tikz/.cd,
  K/.store in=\K,
  K=1  
   }

\author[T. Brysiewicz]{Taylor Brysiewicz}
\address[T. Brysiewicz]{Department of Mathematics, University of Western Ontario, London, Canada (ORCID: 0000-0003-4272-5934)}
\email{tbrysiew@uwo.ca}
\author[F. Gesmundo]{Fulvio Gesmundo}
\address[F. Gesmundo]{Institut de Mathématiques de Toulouse; UMR5219 -- Université de Toulouse; CNRS -- UPS, F-31062 Toulouse Cedex 9, France (ORCID: 0000-0001-6402-021X)}
\email{fgesmund@math.univ-toulouse.fr}

\newcommand{\mycomment}[1]{}

\usepackage[all,cmtip]{xy}

\theoremstyle{definition}
\newtheorem{theorem}{Theorem}[section]

\newtheorem{conjecture}[theorem]{Conjecture}

\newtheorem{proposition}[theorem]{Proposition}

\newcommand{\SO}{{\textrm{SO}}}

\newcommand{\HSO}{{\textrm{HSO}}}
\newcommand{\mydef}[1]{{#1}}

\numberwithin{equation}{section}

\usepackage[breaklinks,bookmarks=false, hidelinks]{hyperref} 
\hypersetup{
    colorlinks,
    linkcolor={red!50!black},
    citecolor={green!50!black},
    urlcolor={blue!80!black}
}

\numberwithin{table}{section}

\usepackage{tikz}
\usetikzlibrary{calc}
\tikzstyle{vertex}=[circle, draw, inner sep=0pt, minimum size=6pt, fill=black]

\renewcommand{\O}{{\textrm{O}}}
\title{Marvelous slices of orthogonal matrices}
\date{}

\usepackage[margin=1in]{geometry}

\begin{document}

\begin{abstract}
The space of $4 \times 4$ special orthogonal matrices with zeros on the diagonal decomposes into the union of $14$ irreducible surfaces whose intersections are beautifully encoded  by the cuboctahedron. Using this decomposition, we exhibit a totally real witness set for $\SO(4)$. We explain how to obtain a similar decomposition for $\SO(5)$, where the $64$ components can be grouped to obtain such a correspondence with the face lattice of a $3$-polytope. We show that no such pattern exists for $\SO(6)$.
\end{abstract}

\maketitle

\vspace*{-15pt}
\section{Main Result}
In this note, we observe the beautiful geometry of $4 \times 4$  \textit{hollow} special orthogonal matrices
\[
\mydef{\textrm{HSO}(4)} = \{M \in \textrm{Mat}_{4 \times 4} (\mathbb{R}) \mid \textrm{diag}(M) = (0,0,0,0) \text{ and }M \in \textrm{SO}(4) \} \subseteq \mathbb{R}^{4 \times 4} = \mathbb{R}^{16}.
\]
Our point of departure is a full description of the components of $\HSO(4)$ and their intersections.
\begin{theorem}
\label{thm:maintheorem}
The variety $\textrm{HSO}(4) \subseteq \mathbb{R}^{16}$ is the union of $14$ irreducible surfaces: six tori of degree four and eight spheres of degree two. There exists a bijection between these components and the facets of the cuboctahedron $\mathcal C$ (illustrated on the left of \autoref{fig:polytopes}) satisfying
\begin{enumerate}
\item The six tori $\SO(2) \times \SO(2) \cong \mathbb{S}^1 \times \mathbb{S}^1$ correspond to the six quadrilateral faces of $\mathcal C$.
\item The eight spheres $\mathbb{S}^2$ correspond to the eight triangular faces of $\mathcal C$.
\item Components intersect in a curve if and only if the corresponding facets share an edge. 
\item Components intersect in points if and only if the corresponding facets share a vertex. 
\item Curve intersections of components are copies of $\SO(2) \cong \mathbb{S}^1$.
\item Point intersections  are signed permutation matrices and their negatives, i.e. $\mathbb{S}^0$. 
\end{enumerate} 
The relevant bijection is illustrated in \autoref{fig:CuboctahedronFacets}.
\end{theorem}

\begin{figure}[ht!]
\begin{tikzpicture}[x={(1.4142cm, 0.8165cm)}, y={(-1.4142cm, 0.8165cm)}, z={(0.0000cm, 1.6330cm)}, scale=1.2*0.75, line join=round]
\coordinate (v0) at (-1.0000, -1.0000, 0.0000);
\coordinate (v1) at (-1.0000, 1.0000, 0.0000);
\coordinate (v2) at (1.0000, -1.0000, 0.0000);
\coordinate (v3) at (1.0000, 1.0000, 0.0000);
\coordinate (v4) at (-1.0000, 0.0000, -1.0000);
\coordinate (v5) at (-1.0000, 0.0000, 1.0000);
\coordinate (v6) at (1.0000, 0.0000, -1.0000);
\coordinate (v7) at (1.0000, 0.0000, 1.0000);
\coordinate (v8) at (0.0000, -1.0000, -1.0000);
\coordinate (v9) at (0.0000, -1.0000, 1.0000);
\coordinate (v10) at (0.0000, 1.0000, -1.0000);
\coordinate (v11) at (0.0000, 1.0000, 1.0000);
\definecolor{polycolor}{RGB}{250, 51, 41} 
\fill[polycolor, opacity=0.8] (v6) -- (v10) -- (v3) -- cycle;
\fill[polycolor, opacity=0.8] (v7) -- (v6) -- (v3) -- cycle;
\fill[polycolor, opacity=0.8] (v6) -- (v10) -- (v4) -- cycle;
\fill[polycolor, opacity=0.8] (v11) -- (v10) -- (v3) -- cycle;
\fill[polycolor, opacity=0.8] (v10) -- (v4) -- (v1) -- cycle;
\fill[polycolor, opacity=0.8] (v6) -- (v8) -- (v4) -- cycle;
\fill[polycolor, opacity=0.8] (v7) -- (v6) -- (v2) -- cycle;
\fill[polycolor, opacity=0.8] (v11) -- (v10) -- (v1) -- cycle;
\fill[polycolor, opacity=0.8] (v6) -- (v8) -- (v2) -- cycle;
\fill[polycolor, opacity=0.8] (v11) -- (v7) -- (v3) -- cycle;
\fill[polycolor, opacity=0.8] (v8) -- (v4) -- (v0) -- cycle;
\fill[polycolor, opacity=0.8] (v11) -- (v5) -- (v1) -- cycle;
\fill[polycolor, opacity=0.8] (v9) -- (v7) -- (v2) -- cycle;
\fill[polycolor, opacity=0.8] (v5) -- (v4) -- (v1) -- cycle;
\fill[polycolor, opacity=0.8] (v9) -- (v8) -- (v2) -- cycle;
\fill[polycolor, opacity=0.8] (v11) -- (v9) -- (v7) -- cycle;
\fill[polycolor, opacity=0.8] (v9) -- (v8) -- (v0) -- cycle;
\fill[polycolor, opacity=0.8] (v11) -- (v9) -- (v5) -- cycle;
\fill[polycolor, opacity=0.8] (v5) -- (v4) -- (v0) -- cycle;
\fill[polycolor, opacity=0.8] (v9) -- (v5) -- (v0) -- cycle;
\draw[black, thick]
(v1) -- (v10)
(v4) -- (v10)
(v2) -- (v6)
(v6) -- (v8)
(v3) -- (v10)
(v3) -- (v6)
(v6) -- (v10)
(v3) -- (v7)
(v3) -- (v11)
(v2) -- (v8)
(v1) -- (v4)
(v1) -- (v11)
(v7) -- (v11)
(v2) -- (v7)
(v4) -- (v8)
;
\draw[black, dashed]
(v0) -- (v4)
(v0) -- (v8)
(v0) -- (v5)
(v0) -- (v9)
(v5) -- (v9)
(v2) -- (v9)
(v7) -- (v9)
(v1) -- (v5)
(v5) -- (v11)
;
\end{tikzpicture} \quad \quad  \quad \quad 
\tdplotsetmaincoords{0.00}{-90.00}
\begin{tikzpicture}[tdplot_main_coords,baseline=-2.6cm, xshift=1cm, scale=1.1*0.75, yshift = -0.4cm, line join=round]
\coordinate (d0) at (1.8801, -0.0000, -1.6048);
\coordinate (d1) at (0.0000, 1.8801, -1.6048);
\coordinate (d2) at (-1.3294, 1.3294, 1.6048);
\coordinate (d3) at (-1.3294, -1.3294, 1.6048);
\coordinate (d4) at (1.3294, 1.3294, 1.6048);
\coordinate (d5) at (1.3294, -1.3294, 1.6048);
\coordinate (d6) at (0.0000, -1.8801, -1.6048);
\coordinate (d7) at (-1.8801, 0.0000, -1.6048);
\coordinate (d8) at (0.0000, 2.9814, 1.0541);
\coordinate (d9) at (2.1082, 2.1082, -1.0541);
\coordinate (d10) at (-0.0000, -2.9814, 1.0541);
\coordinate (d11) at (2.1082, -2.1082, -1.0541);
\coordinate (d12) at (-2.1082, 2.1082, -1.0541);
\coordinate (d13) at (-2.9814, 0.0000, 1.0541);
\coordinate (d14) at (-0.0000, -0.0000, 2.0479);
\coordinate (d15) at (2.9814, 0.0000, 1.0541);
\coordinate (d16) at (-2.1082, -2.1082, -1.0541);
\coordinate (d17) at (0.0000, 0.0000, -2.0479);
\definecolor{polycolor}{RGB}{154, 100, 246} 
\fill[polycolor, opacity=0.8] (d7) -- (d6) -- (d17) -- cycle;
\fill[polycolor, opacity=0.8] (d0) -- (d6) -- (d17) -- cycle;
\fill[polycolor, opacity=0.8] (d1) -- (d7) -- (d17) -- cycle;
\fill[polycolor, opacity=0.8] (d1) -- (d0) -- (d17) -- cycle;
\fill[polycolor, opacity=0.8] (d7) -- (d6) -- (d16) -- cycle;
\fill[polycolor, opacity=0.8] (d0) -- (d6) -- (d11) -- cycle;
\fill[polycolor, opacity=0.8] (d1) -- (d7) -- (d12) -- cycle;
\fill[polycolor, opacity=0.8] (d1) -- (d0) -- (d9) -- cycle;
\fill[polycolor, opacity=0.8] (d6) -- (d11) -- (d10) -- cycle;
\fill[polycolor, opacity=0.8] (d6) -- (d16) -- (d10) -- cycle;
\fill[polycolor, opacity=0.8] (d7) -- (d16) -- (d13) -- cycle;
\fill[polycolor, opacity=0.8] (d7) -- (d12) -- (d13) -- cycle;
\fill[polycolor, opacity=0.8] (d0) -- (d11) -- (d15) -- cycle;
\fill[polycolor, opacity=0.8] (d0) -- (d9) -- (d15) -- cycle;
\fill[polycolor, opacity=0.8] (d1) -- (d12) -- (d8) -- cycle;
\fill[polycolor, opacity=0.8] (d1) -- (d9) -- (d8) -- cycle;
\fill[polycolor, opacity=0.8] (d5) -- (d11) -- (d10) -- cycle;
\fill[polycolor, opacity=0.8] (d5) -- (d11) -- (d15) -- cycle;
\fill[polycolor, opacity=0.8] (d3) -- (d16) -- (d10) -- cycle;
\fill[polycolor, opacity=0.8] (d3) -- (d16) -- (d13) -- cycle;
\fill[polycolor, opacity=0.8] (d4) -- (d9) -- (d15) -- cycle;
\fill[polycolor, opacity=0.8] (d4) -- (d9) -- (d8) -- cycle;
\fill[polycolor, opacity=0.8] (d2) -- (d12) -- (d13) -- cycle;
\fill[polycolor, opacity=0.8] (d2) -- (d12) -- (d8) -- cycle;
\fill[polycolor, opacity=0.8] (d3) -- (d5) -- (d10) -- cycle;
\fill[polycolor, opacity=0.8] (d4) -- (d5) -- (d15) -- cycle;
\fill[polycolor, opacity=0.8] (d2) -- (d3) -- (d13) -- cycle;
\fill[polycolor, opacity=0.8] (d2) -- (d4) -- (d8) -- cycle;
\fill[polycolor, opacity=0.8] (d3) -- (d5) -- (d14) -- cycle;
\fill[polycolor, opacity=0.8] (d4) -- (d5) -- (d14) -- cycle;
\fill[polycolor, opacity=0.8] (d2) -- (d3) -- (d14) -- cycle;
\fill[polycolor, opacity=0.8] (d2) -- (d4) -- (d14) -- cycle;
\draw[black, dashed]
(d6) -- (d11)
(d6) -- (d16)
(d7) -- (d16)
(d6) -- (d17)
(d7) -- (d17)
(d7) -- (d12)
(d0) -- (d11)
(d0) -- (d17)
(d0) -- (d9)
(d1) -- (d17)
(d1) -- (d12)
(d1) -- (d9)
;
\draw[black, thick]
(d10) -- (d11)
(d10) -- (d16)
(d13) -- (d16)
(d12) -- (d13)
(d11) -- (d15)
(d9) -- (d15)
(d8) -- (d12)
(d8) -- (d9)
(d5) -- (d10)
(d5) -- (d15)
(d3) -- (d10)
(d3) -- (d13)
(d5) -- (d14)
(d3) -- (d14)
(d4) -- (d15)
(d4) -- (d8)
(d4) -- (d14)
(d2) -- (d13)
(d2) -- (d8)
(d2) -- (d14)
;
\end{tikzpicture}
\caption{The cuboctahedron $\mathcal C$ (left) and another polytope $\mathcal P$ (right). Their facets are in bijection with components of $\textrm{HSO}(4)$ and quadruples of components of $\textrm{SO}^\star(5)$ respectively. The face lattice of each gives the intersection lattice of the corresponding components in agreement with \autoref{thm:maintheorem} and \autoref{thm:SO5}.} 
\label{fig:polytopes}
\end{figure}
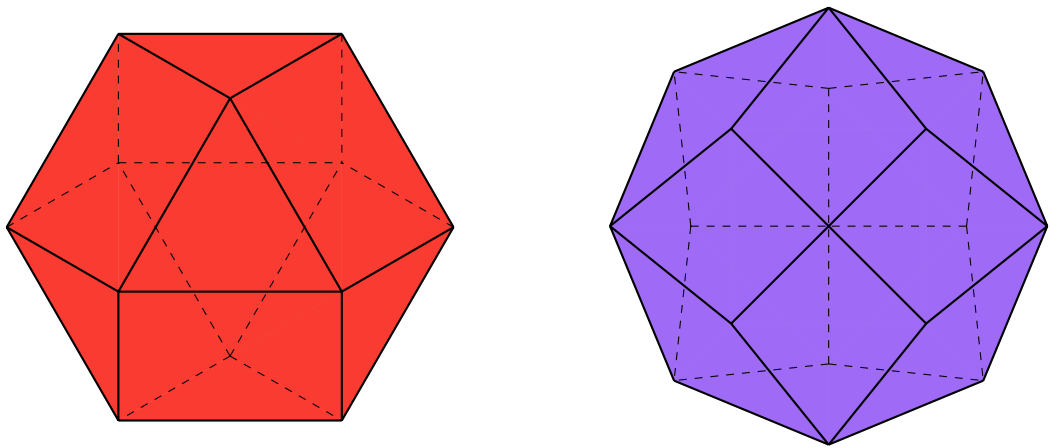

For $\textrm{SO}(5)$, the story is similar. There is a zero pattern which slices $\textrm{SO}(5)$ down to a union $\textrm{SO}^\star(5)$ of surfaces. As with $\textrm{HSO}(4)$, this slice is \textit{degree-generic} in the sense that the sum of the degrees of the $64$ components equals the degree $384$ of $\textrm{SO}(5)$ \cite{SOn}. We put quadruples of components of $\textrm{SO}^\star(5)$ in bijection with the facets of the polytope on the right of \autoref{fig:polytopes} in such a way that an analogue of \autoref{thm:maintheorem} holds. The full statement is given in \autoref{thm:SO5}.

In \autoref{sec:surfaces} we describe the $14$ surfaces of $\HSO(4)$ and establish the bijection between these surfaces and the facets of the cuboctahedron $\mathcal C$. In \autoref{secsec:curves} and \autoref{secsec:points}, respectively, we describe the one-dimensional and zero-dimensional intersections of these components. In \autoref{sec:witness} we use this decomposition to find a totally real witness set for $\SO(4)$, resolving the conjecture \cite[Conjecture 7.1]{SOn} for $n=4$. We extend this story to $n=5$ in \autoref{sec:SO5}: we identify $64$ components of $\textrm{SO}^\star(5)$ which are in bijection with the $16$ facets of the polytope $\mathcal P$ shown in the right of \autoref{fig:polytopes} when grouped by fours. Just as in the case of $n=4$, the face lattice of this polytope encodes the intersections between groups of components, as detailed in \autoref{thm:SO5}. We show the pattern does not continue for $n=6$, see \autoref{thm:SO6}. All results in this paper, other than \autoref{thm:SO6}, can be verified by hand. For convenience, we provide code to verify all results using computer algebra software  \cite{Macaulay2,Oscar,OscarBook,Polymake,Nauty}, available at \cite{BGmarvelousRepo}.

\subsection{Related work and motivation} Hollow matrices appear in several areas: we refer to \cite{gentle2007matrix} for some examples. This includes the case of adjacency matrices of graphs, and that of dissimilarity matrices in statistics \cite{trosset2002extensions}. Hollow orthogonal matrices arise when looking for normal forms of non-definite operators with desirable stability properties \cite{Nic} and in other combinatorial frameworks for which we refer to \cite{Colbourn2010}. 

Our investigation began at the $2018$ ICERM semester on nonlinear algebra. We sought to generalize the formula for the degree of the special orthogonal group in \cite{SOn} to Stiefel manifolds. One strategy we considered was to compute the degrees of these homogeneous spaces by predicting their intersection with select linear spaces. The goal was to re-establish the $\textrm{SO}(n)$ degree formula via enumerative combinatorics, rather than an application of Kazarnovskii's theorem \cite{Kazarnovskii}. This proved incredibly successful for $\textrm{SO}(4)$ (as described in \autoref{thm:maintheorem}) since the hollow linear space decomposes $\textrm{SO}(4)$ combinatorially. This program is successful for $\textrm{SO}(5)$ as well (as described in \autoref{sec:SO5}). This idea, however, does not succeed to provide a combinatorial interpretation of the degree of $\textrm{SO}(6)$ (see \autoref{thm:SO6}). Ultimately, the goal of determining the degrees of Stiefel manifolds was achieved using alternative techniques \cite{StiefelManifolds}, leaving the \textit{marvelous} phenomenon witnessed during this process to be shared in the present work.

 \autoref{thm:maintheorem} is used to provide a positive answer to the next unknown instance of a conjecture \cite[Conjecture 7.1]{SOn} concerning the real algebraic geometry of linear algebraic groups: does there exist a zero-dimensional linear slice of $\textrm{SO}(n)$ consisting entirely of real points? The answer is `yes' and the $\textrm{deg}(\textrm{SO}(4)) = 40$ real points are given explicitly in \autoref{sec:witness}. 

Our main motivations for this manuscript are as follows. First and foremost, we wanted to share these examples of $\textrm{HSO}(4)$, and its generalization $\SO^\star(5)$ with the mathematical community;  we could not find them in the literature and they pertain to mathematical objects of broad general interest. Secondly, we hope that our work will encourage investigation of the intrinsic geometric reason for why such decomposable slices exist for $\textrm{SO}(4)$ and $\textrm{SO}(5)$ and whether the phenomenon persists in some form for $n>5$. Finally, we hope this serves as a motivating example to real algebraic geometers, showcasing the heuristic that structured combinatorial decompositions can make real algebro-geometric searches feasible.

\subsection*{Acknowledgements}
We are thankful to Michael Joswig for his helpful input regarding the polytope $\mathcal P$. The initial discoveries in this paper were obtained during the $2018$ ICERM semester on \emph{Nonlinear Algebra} (NSF DMS-1439786).
TB is currently supported by an NSERC Discovery Grant (RGPIN-2023-03551).

\section{The fourteen surfaces of hollow $\textrm{SO}(4)$}
\label{sec:surfaces}
We describe the components of $\textrm{HSO}(4)$ and give the bijection with the facets of the cuboctahedron.

\subsection{Eight Spheres} We define eight spheres in $\mathbb{R}^{16}$, all of which are parametrized by points $(x,y,z)$ on the standard unit sphere $\mathbb{S}^2 \subset \mathbb{R}^3$. One sphere, denoted $S_{+++}^+$, is the image of the map 
\begin{equation}
\label{eq:Spppp}\mydef{S_{+++}^+} \quad \quad \quad \quad \quad 
\mathbb{S}^2 \ni (x,y,z) \mapsto {\tiny \begin{bmatrix} 0 & x & -y & z \\ x & 0 & z & y \\ y & z & 0 & -x \\ z & -y & -x & 0 \end{bmatrix}}.
\end{equation}
This parametrization is summarized by the  diagram
\begin{figure}[!htpb]
\includegraphics[scale=1]{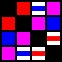} \quad   \raisebox{20pt}{ \quad where \quad $({\color{red}{\blacksquare}},{\color{blue}{\blacksquare}},{\color{diagrammagenta}{\blacksquare}}) \in \mathbb{S}^2$, \quad $\blacksquare=0$, \quad $({\color{red}{\rect}},{\color{blue}{\rect}},{\color{diagrammagenta}{\rect}}) = -({\color{red}{\blacksquare}},{\color{blue}{\blacksquare}},{\color{diagrammagenta}{\blacksquare}})$.}
\end{figure}

\noindent The variety $S_{+++}^+$ is an irreducible surface of degree $2$ and by inspection, $S_{+++}^+ \subseteq \textrm{HSO}(4)$. 

There are seven other components, each indexed by an even $4$-signature: \[S_{-++}^-, S_{+-+}^-, S_{++-}^-,{S_{--+}^+},S_{-+-}^+,S_{+--}^+,S_{---}^-.\]
The lower three signs correspond to negation of columns $2,3,$ and $4$ of \eqref{eq:Spppp}, respectively, whereas the top sign indicates a negation of the first row. This top sign may be viewed as notational redundancy since its value must be chosen to make the determinant one instead of negative one. For example, the component $S_{++-}^-$ is obtained via $S_{+++}^+$ by negating the fourth column and first row:

\begin{figure}[!htpb]
\raisebox{20pt}{$ \mydef{S_{++-}^-} \quad \quad \quad  \quad 
\mathbb{S}^2 \ni (x,y,z) \mapsto {\small \begin{bmatrix} 0 & -x & y & z \\ x & 0 & z & -y \\ y & z & 0 & x \\ z & -y & -x & 0 \end{bmatrix}}$ }\quad \quad \includegraphics[scale=1]{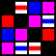}.
\end{figure}

\subsection{Six Tori} There are three ways to partition a hollow $4 \times 4$ matrix into two $2 \times 2$ blocks:
\begin{equation}
\label{eq:blockpatterns}{\tiny
T^{(12)(34)} = 
\begin{bmatrix} 0 & 0 & \Delta & \Delta \\
0 & 0 & \Delta & \Delta \\
\square & \square & 0 & 0 \\
\square & \square & 0 & 0 
\end{bmatrix}, \quad \quad 
T^{(13)(24)} = 
\begin{bmatrix} 0 & \square &0   & \square \\
\Delta & 0 & \Delta & 0 \\
0 & \square & 0 & \square \\
\Delta & 0 & \Delta & 0 
\end{bmatrix}, \quad \quad 
T^{(14)(23)} = 
\begin{bmatrix} 0 & \square &\square   & 0 \\
\Delta & 0 & 0 & \Delta \\
\Delta & 0 & 0 & \Delta \\
0 & \square & \square & 0 
\end{bmatrix}}.
\end{equation}
Each $2 \times 2$ submatrix marked by either $\Delta$ or $\square$ can hold a copy of an orthogonal $2 \times 2$ matrix of common determinant $\pm 1$.  This gives explicit parametrizations for six tori,
\[T_{+}^{(12)(34)},T_{-}^{(12)(34)},T_{+}^{(13)(24)},T_{-}^{(13)(24)},T_{+}^{(14)(23)},T_{-}^{(14)(23)},\] via a copy of $\SO(2) \times \SO(2) \cong \mathbb{S}^1 \times \mathbb{S}^1$.  For example $T_{+}^{(12)(34)}$ is parametrized as
\begin{center}
$
\mydef{T^{(12)(34)}_{+}}    \quad   \quad   \quad   \quad \mathbb{S}^1 \times \mathbb{S}^1 \ni ((x,y),(z,w)) \mapsto {\tiny \begin{bmatrix}
0 & 0 & z & -w \\
0 & 0 & w & z \\
x & -y & 0 & 0 \\
y & x & 0 & 0 \end{bmatrix} }$
\end{center}

\noindent and is summarized diagrammatically as
\begin{center}
\includegraphics[scale=0.8]{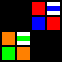} \quad \quad  \raisebox{18pt}{ \text{where} \quad \quad ${\small{\begin{array}{c}({\color{orange}{\blacksquare}},{\color{green}{\blacksquare}}),({\color{red}{\blacksquare}},{\color{blue}{\blacksquare}}) \in \mathbb{S}^1, \quad \blacksquare = 0 \\ ({\color{orange}{\rect}},{\color{green}{\rect}},{\color{red}{\rect}},{\color{blue}{\rect}})=
-({\color{orange}{\blacksquare}},{\color{green}{\blacksquare}},{\color{red}{\blacksquare}},{\color{blue}{\blacksquare}}).
\end{array}}}$}
\end{center}

\noindent The other five tori have similar diagrammatic representations. Each torus has degree four. Since $\textrm{deg}(\SO(4))=40$ \cite{SOn} we have that
\[
\textrm{deg}(\SO(4)) = 40 = 8 \cdot 2 + 6 \cdot 4 = 8\cdot\left(\deg(S_{\pm \pm \pm}^\pm)\right) + 6 \cdot \left(\deg\left(T^{(ab)(cd))}_{\pm}\right)\right)  .
\]
The orthogonal group $\SO(4)$ is a smooth variety. Therefore, linear sections with the expected codimension cannot have degree higher than $\deg \SO(4)$; see, e.g., \cite[Corollary 2.5]{3264}. This guarantees that the $14$ surfaces we have exhibited comprise all components of $\textrm{HSO}(4)$.

Now that all components of $\textrm{HSO}(4)$ have been identified, we give the bijection of \autoref{thm:maintheorem} between the $14$ facets of the cuboctahedron $\mathcal C$ and these $14$ components pictorially in \autoref{fig:CuboctahedronFacets}. 

\begin{figure}[!htpb]
\includegraphics[scale=0.48]{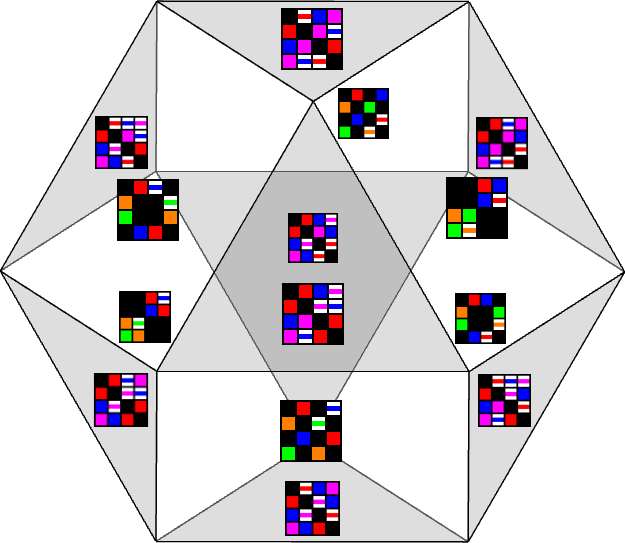}
\caption{The cuboctahedron with facets labeled by the diagrams identifying components of $\textrm{HSO}(4)$. This is the relevant bijection for \autoref{thm:maintheorem}.}
\label{fig:CuboctahedronFacets}
\end{figure}

\section{Twenty-four curves and twelve pairs of points}
\subsection{Twenty-four circles}
\label{secsec:curves}
The $14$ surfaces of $\HSO(4)$ intersect in dimension one according to \autoref{tab:incidence}.
\begin{table}[!htpb]
$${\footnotesize{\begin{array}{|c||c|c|c|c|c|c|}  \hline 
& T^{(12)(34)}_+ & T^{(12)(34)}_- & T^{(13)(24)}_+ & T^{(13)(24)}_- & T^{(14)(23)}_+ & T^{(14)(23)}_- \\ \hline \hline 
S_{+++}^+ & & \times & & \times  & &  \times \\
S_{-++}^- & \times  & & &  \times & &  \times \\
S_{+-+}^- & & \times  &  \times & & &  \times \\
S_{++-}^- & &  \times & &  \times &  \times & \\
S_{--+}^+ &  \times & &  \times & & & \times  \\
S_{-+-}^+ &  \times & & & \times  &  \times & \\
S_{+--}^+ & & \times  & \times  & &  \times & \\
S_{---}^- & \times  & & \times  & & \times  & \\ \hline 
\end{array}}}$$
\caption{The incidence matrix of surface components of $\HSO(4)$ that intersect in a curve.}
\label{tab:incidence}
\end{table}
In particular, the tori never intersect each other in dimension one, nor do the spheres. Since the action of even column/row negation acts transitively on the spheres, we focus on analyzing the first row of the table, corresponding to $S_{+++}^+$. Imposing the three $2+2$ block patterns of \eqref{eq:blockpatterns} on the matrix \eqref{eq:Spppp} enforces $x,y,$ or $z$ (respectively) equal to zero, resulting in the matrix formats
\begin{equation}
\label{eq:circlematrixformats}
{\small{T^{(12)(34)}: {\tiny \begin{bmatrix}
0 & 0 & -y & z \\ 
0 & 0 & z & y \\
y & z & 0 & 0 \\
z & -y & 0 & 0 
\end{bmatrix}}, \quad \,\,\,  \,\,\, 
T^{(13)(24)}: {\tiny \begin{bmatrix}
0 & x & 0 & z \\ 
x & 0 & z & 0 \\
0 & z & 0 & -x \\
z &0 & -x & 0 
\end{bmatrix}}, \quad  \,\,\,  \,\,\,  
T^{(14)(23)}: {\tiny \begin{bmatrix}
0 & x & -y & 0 \\ 
x & 0 & 0 & y \\
y & 0 & 0 & -x \\
0 & -y & -x & 0 
\end{bmatrix}}}}.
\end{equation}
Each parametrization in \eqref{eq:circlematrixformats} gives the diagonal embedding $\mathbb{S}^1 \hookrightarrow \mathbb{S}^1 \times \mathbb{S}^1 \cong T^{(ab)(cd)}_{\pm}$. The determinants of each $2 \times 2$ block of each of these circles are observed to be $-1$, verifying \autoref{tab:incidence}.

The incidence described in \autoref{tab:incidence} of sphere-torus intersections of dimension one coincides with the incidence of triangle-quadrilateral intersections along edges of the cuboctahedron as in \autoref{fig:CuboctahedronFacets}. In particular, the $14$ surface components of $\HSO(4)$ intersect in dimension one along $24$ circles.

\subsection{Twelve pairs of fixed point free signed permutation matrices}
\label{secsec:points}
Having treated the two-dimensional components of $\HSO(4)$ and their one-dimensional intersections, we now turn toward their zero-dimensional intersections. As one would hope, these are encoded in the facet-vertex incidence of the cuboctahedron. The analysis of the previous section shows that sphere-torus pairs which do not intersect in curves, do not intersect at all. Thus, the only remaining possible non-empty intersections of components are between sphere-sphere pairs or torus-torus pairs.

The possible zero patterns occurring by intersecting the torus block patterns \eqref{eq:blockpatterns} are 
\begin{equation*}
\begin{tikzpicture}[scale=0.8]
\node[] (A) at (0, 0) {$  \hspace{-30pt}{\color{white}{P^{(13)(24)}=}} {\tiny{\begin{bmatrix} 0 & 0 & \star  & 0 \\
0 & 0 & 0 & \star  \\
\star  & 0 & 0 & 0 \\
0 & \star  & 0 & 0 \end{bmatrix}}}$
};
\node[] (B) at (6, 0) {${\tiny{ \begin{bmatrix} 0 & 0 & 0  & \star \\
0 & 0 & \star & 0 \\
0  & \star & 0 & 0 \\
\star & 0  & 0 & 0 \end{bmatrix}}} {\color{white}{=P^{(14)(23)}}} $
};
\node[] (C) at (3, 3.46) {$ \hspace{-60pt}{\color{white}{P^{(12)(34)} =}} {\tiny{\begin{bmatrix} 0 & \star & 0 & 0 \\
\star & 0 & 0 & 0 \\
0 & 0 & 0 & \star \\
0 & 0 & \star & 0 \end{bmatrix}}}$
};

\draw (A) -- (B) node[midway, above]{$T^{(12)(34)}$};
\draw (B) -- (C) node[midway, right]{$ \hspace{10pt}T^{(13)(24)}$};
\draw (C) -- (A) node[midway, left]{$T^{(14)(23)}$};
\end{tikzpicture}
\end{equation*}
\noindent Matrices of these patterns which appear in $\HSO(4)$ are those which have the $\star$ symbols replaced with an even number of $1$'s and $-1$'s. They may be thought of as copies of $\mathbb{S}^0 \cong \mathrm{O}(1)$ parametrized via the illustration in \autoref{fig:CuboctahedronVertices}. There are $3  \cdot 8 \cdot \frac 1 2=12$ of them. These comprise the non-empty sphere-sphere intersections as well, and they appear as intersections in agreement with \autoref{fig:CuboctahedronFacets}. This fact completes the proof of \autoref{thm:maintheorem}.

\begin{figure}[!htpb]
\includegraphics[scale=0.47]{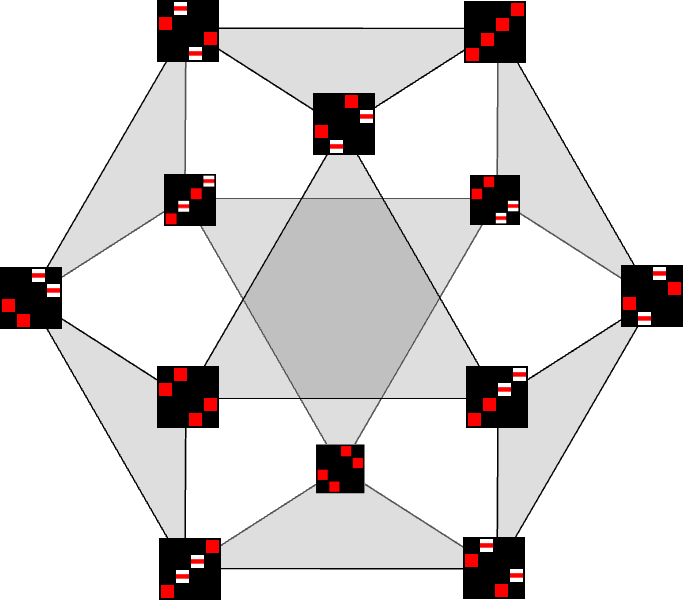}
\caption{The vertices of $\mathcal C$ correspond to zero-dimensional intersections of components of $\textrm{HSO}(4)$. Each represents a signed double transposition and its negative.}
\label{fig:CuboctahedronVertices}
\end{figure}

Each zero-dimensional intersection is a signed permutation matrix, paired with its negative. Since the matrix must be hollow, these permutations have no fixed points; however, not all fixed point free permutations arise. Only the double transpositions show up, whereas the six $4$-cycles do not.  
\section{A totally real witness set}
\label{sec:witness}
The decomposition of $\textrm{HSO}(4)$ is relevant to real numerical algebraic geometry \cite{NAG} and semidefinite programming. In numerical algebraic geometry, varieties are represented on a computer as \textit{witness sets}: zero-dimensional intersections with generic linear spaces of complementary dimension. If all such intersection points are real, we say the witness set is \textit{totally real}. An open problem \cite[Conjecture 7.1]{SOn} is to determine if $\SO(n)$ admits a totally real witness set for all $n$. This conjecture relates to a low-rank semidefinite programming optimization problem, as described in \cite{SOn,LowRankSDP}. 
 For an irreducible variety $X \subseteq \mathbb{C}^n$ and generic hyperplane $\mathcal H$, Bertini-like theorems (e.g. \cite[Theorem 6.3]{BertiniBook}, \cite[Section 11]{HarrisAG}) tell us what to expect of $X \cap \mathcal H$:
\begin{enumerate}
\item (Degree) The degree of $\mathcal H \cap X$ is the degree of $X$.  
\item (Dimension) The dimension of $\mathcal H \cap X$ is $\dim(X)-1$.  
\item (Irreducible) If $\dim(X)>1$ then $\mathcal H \cap X$ is irreducible.  
\end{enumerate}
 We refer to a hyperplane section as \textit{degree-generic}, \textit{dimension-generic}, and \textit{irreducible-generic} accordingly. These definitions extend to linear sections of higher codimension. In the context of numerical algebraic geometry and witness sets, we require that $X \cap \mathcal L$ is obtained via a sequence of hyperplane sections which are both degree-generic and dimension-generic at each step. The previous section shows that the hollow hyperplanes $\{x_{i,i}=0\}_{i=1}^4$ work as the first four slices for $\textrm{SO}(4)$. Note that irreducible-genericity fails spectacularly.

The benefit of studying structured hyperplane sections of a variety is that, often,  such sections are easy to control and understand. The counterpoint to this benefit is that the more special a linear section is, the less likely it is to be degree, dimension, or irreducible-generic. Incredibly, the hollow linear space of codimension $4$ strikes a perfect balance for $\SO(4)$ between genericity and speciality. It allows us to search through the space of codimension-two linear sections of $\HSO(4)$ for a pair of hyperplanes which cut $\HSO(4)$ down to $40$ \textit{real} points. Here are two:
\begin{align*}{\tiny{
H_1: \begin{bmatrix}
0 & 1 & -1 & 0 \\
-1 & 0 & 0 & -1 \\ 
1 & 0 & 0 & 1 \\
0 & -1 & 1 & 0
\end{bmatrix} \cdot X = 0 \quad \quad \quad \quad \quad \quad 
H_2: \begin{bmatrix}
0 & 1 & 0 & -1 \\
-1 & 0 & -1 & 0 \\ 
0 & -1 & 0 & 1 \\
1 & 0 & 1 & 0
\end{bmatrix} \cdot X =0 }},
\end{align*}
where $\cdot$ denotes the entry-wise inner product. The $40$ real intersection points are 
\[
{\scriptsize {
\begin{bmatrix}
0 & \text{-}a & a & a \\
\text{-}a & 0 & \text{-}a & a \\
a & a & 0 & a \\
\text{-}a & a & a & 0
\end{bmatrix}\hspace{-5pt}
\begin{bmatrix}
0 & a & \text{-}a & a \\
a & 0 & a & a \\
\text{-}a & \text{-}a & 0 & a \\
\text{-}a & a & a & 0
\end{bmatrix}\hspace{-5pt}
\begin{bmatrix}
0 & \text{-}a & a & a \\
\text{-}a & 0 & a & \text{-}a \\
\text{-}a & a & 0 & a \\
a & a & a & 0
\end{bmatrix}\hspace{-5pt}
\begin{bmatrix}
0 & a & \text{-}a & a \\
\text{-}a & 0 & a & a \\
\text{-}a & a & 0 & \text{-}a \\
a & a & a & 0
\end{bmatrix} \hspace{-5pt}
\begin{bmatrix}
0 & a & a & \text{-}a \\
a & 0 & a & a \\
\text{-}a & a & 0 & a \\
\text{-}a & \text{-}a & a & 0
\end{bmatrix}\hspace{-5pt}
\begin{bmatrix}
0 & a & a & a \\
\text{-}a & 0 & \text{-}a & a \\
\text{-}a & a & 0 & \text{-}a \\
\text{-}a & \text{-}a & a & 0
\end{bmatrix}\hspace{-5pt}
\begin{bmatrix}
0 & \text{\text{-}}a & \text{\text{-}}a & \text{\text{-}}a \\
a & 0 & \text{\text{-}}a & a \\
a & a & 0 & \text{\text{-}}a \\
a & \text{\text{-}}a & a & 0
\end{bmatrix}\hspace{-5pt}
\begin{bmatrix}
0 & \text{\text{-}}a & \text{\text{-}}a & a \\
a & 0 & \text{\text{-}}a & \text{\text{-}}a \\
\text{\text{-}}a & \text{\text{-}}a & 0 & \text{\text{-}}a \\
a & \text{\text{-}}a & a & 0
\end{bmatrix}
}}
\]
\[{\scriptsize{
\begin{bmatrix}
0 & 0 & \text{-}1 & 0 \\
0 & 0 & 0 & 1 \\
0 & \text{-}1 & 0 & 0 \\
\text{-}1 & 0 & 0 & 0
\end{bmatrix}\hspace{-1pt}
\begin{bmatrix}
0 & 0 & 1 & 0 \\
0 & 0 & 0 & \text{-}1 \\
0 & \text{-}1 & 0 & 0 \\
\text{-}1 & 0 & 0 & 0
\end{bmatrix}\hspace{-1pt}
\begin{bmatrix}
0 & \text{-}1 & 0 & 0 \\
0 & 0 & \text{-}1 & 0 \\
0 & 0 & 0 & 1 \\
\text{-}1 & 0 & 0 & 0
\end{bmatrix}\hspace{-1pt}
\begin{bmatrix}
0 & 1 & 0 & 0 \\
0 & 0 & \text{-}1 & 0 \\
0 & 0 & 0 & \text{-}1 \\
\text{-}1 & 0 & 0 & 0
\end{bmatrix}\hspace{-1pt}
\begin{bmatrix}
0 & 0 & \text{-}1 & 0 \\
0 & 0 & 0 & 1 \\
0 & 1 & 0 & 0 \\
1 & 0 & 0 & 0
\end{bmatrix} \hspace{-1pt}
\begin{bmatrix}
0 & 0 & 1 & 0 \\
0 & 0 & 0 & \text{-}1 \\
0 & 1 & 0 & 0 \\
1 & 0 & 0 & 0
\end{bmatrix}\hspace{-1pt}
\begin{bmatrix}
0 & \text{-}1 & 0 & 0 \\
0 & 0 & 1 & 0 \\
0 & 0 & 0 & 1 \\
1 & 0 & 0 & 0
\end{bmatrix}\hspace{-1pt}
\begin{bmatrix}
0 & 1 & 0 & 0 \\
0 & 0 & 1 & 0 \\
0 & 0 & 0 & \text{-}1 \\
1 & 0 & 0 & 0
\end{bmatrix}
}}
\]
\[{\scriptsize{
\begin{bmatrix}
0 & 0 & 0 & \text{-}1 \\
0 & 0 & 1 & 0 \\
1 & 0 & 0 & 0 \\
0 & 1 & 0 & 0
\end{bmatrix}\hspace{-1pt}
\begin{bmatrix}
0 & 0 & 0 & 1 \\
0 & 0 & \text{-}1 & 0 \\
1 & 0 & 0 & 0 \\
0 & 1 & 0 & 0
\end{bmatrix}\hspace{-1pt}
\begin{bmatrix}
0 & 0 & \text{-}1 & 0 \\
1 & 0 & 0 & 0 \\
0 & 0 & 0 & 1 \\
0 & 1 & 0 & 0
\end{bmatrix} \hspace{-1pt}
\begin{bmatrix}
0 & 0 & \text{-}1 & 0 \\
\text{-}1 & 0 & 0 & 0 \\
0 & 0 & 0 & \text{-}1 \\
0 & 1 & 0 & 0
\end{bmatrix}\hspace{-1pt}
\begin{bmatrix}
0 & 0 & 0 & \text{-}1 \\
0 & 0 & 1 & 0 \\
\text{-}1 & 0 & 0 & 0 \\
0 & \text{-}1 & 0 & 0
\end{bmatrix}\hspace{-1pt}
\begin{bmatrix}
0 & 0 & 0 & 1 \\
0 & 0 & \text{-}1 & 0 \\
\text{-}1 & 0 & 0 & 0 \\
0 & \text{-}1 & 0 & 0
\end{bmatrix}\hspace{-1pt}
\begin{bmatrix}
0 & 0 & 1 & 0 \\
1 & 0 & 0 & 0 \\
0 & 0 & 0 & 1 \\
0 & \text{-}1 & 0 & 0
\end{bmatrix}\hspace{-1pt}
\begin{bmatrix}
0 & 0 & 1 & 0 \\
\text{-}1 & 0 & 0 & 0 \\
0 & 0 & 0 & \text{-}1 \\
0 & \text{-}1 & 0 & 0
\end{bmatrix}
}}
\]
\[{\scriptsize{
\begin{bmatrix}
0 & \text{-}1 & 0 & 0 \\
0 & 0 & 0 & \text{-}1 \\
\text{-}1 & 0 & 0 & 0 \\
0 & 0 & 1 & 0
\end{bmatrix} \hspace{-1pt}
\begin{bmatrix}
0 & \text{-}1 & 0 & 0 \\
0 & 0 & 0 & 1 \\
1 & 0 & 0 & 0 \\
0 & 0 & 1 & 0
\end{bmatrix}\hspace{-1pt}
\begin{bmatrix}
0 & 0 & 0 & \text{-}1 \\
1 & 0 & 0 & 0 \\
0 & 1 & 0 & 0 \\
0 & 0 & 1 & 0
\end{bmatrix}\hspace{-1pt}
\begin{bmatrix}
0 & 0 & 0 & 1 \\
1 & 0 & 0 & 0 \\
0 & \text{-}1 & 0 & 0 \\
0 & 0 & 1 & 0
\end{bmatrix}\hspace{-1pt}
\begin{bmatrix}
0 & 1 & 0 & 0 \\
0 & 0 & 0 & \text{-}1 \\
\text{-}1 & 0 & 0 & 0 \\
0 & 0 & \text{-}1 & 0
\end{bmatrix}\hspace{-1pt}
\begin{bmatrix}
0 & 1 & 0 & 0 \\
0 & 0 & 0 & 1 \\
1 & 0 & 0 & 0 \\
0 & 0 & \text{-}1 & 0
\end{bmatrix}\hspace{-1pt}
\begin{bmatrix}
0 & 0 & 0 & \text{-}1 \\
\text{-}1 & 0 & 0 & 0 \\
0 & 1 & 0 & 0 \\
0 & 0 & \text{-}1 & 0
\end{bmatrix} \hspace{-1pt}
\begin{bmatrix}
0 & 0 & 0 & 1 \\
\text{-}1 & 0 & 0 & 0 \\
0 & \text{-}1 & 0 & 0 \\
0 & 0 & \text{-}1 & 0
\end{bmatrix}.}}
\]
where $a = \pm \sqrt{1/3}$. Note that each matrix in the first row counts for two points, so $40=2\cdot 8 + 8+8+8$. This resolves \cite[Conjecture 7.1]{SOn} for $n=4$. We state our result for completeness.
\begin{theorem}
There exists a witness set of $\textrm{SO}(4)$ consisting of $40$ real points.
\end{theorem}

This slice was found through the desire for the intersection points to hold combinatorial meaning. Specifically, we considered hyperplanes containing fixed point free signed permutation matrices, other than those in \autoref{fig:CuboctahedronVertices}. The slice above was the most elegant one we found. 

\section{Extending to $\SO(5)$ }
\label{sec:SO5}
One may hope that the hollow version of $\SO(n)$ decomposes in this beautiful way for $n>4$. This is not the case: the variety of hollow $5 \times 5$ special orthogonal matrices is irreducible. Instead, we seek a different zero pattern for $\textrm{SO}(5)$. For $n=4$ the hollow pattern cuts the dimension of $\textrm{SO}(4)$ from six to two. This is the rank of $\textrm{SO}(4)$, that is, the dimension of its maximal torus. 
By an exhaustive calculation of all zero patterns with  $\dim(\SO(5)) - \textrm{rank}(\SO(5)) = 10-2=8$ zeros, modulo symmetry, we found there is a \textit{unique} dimension-generic and degree-generic pattern:
\begin{align} 
\label{eq:5zeropattern}
\raisebox{-28pt}{\includegraphics[scale=0.4]{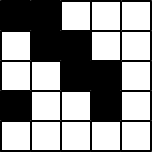}.\tag{$\star$}}
\end{align}
In fact, the exhaustive search shows that there is no other such pattern over the complex numbers either. We write  $\SO^\star(5)$ for the set of $5 \times 5$ special orthogonal matrices with the zero pattern \eqref{eq:5zeropattern}. Just like $\textrm{HSO}(4)$, the space $\SO^\star(5)$ decomposes dramatically.  
We write $\SO^\circ(3)$ for a copy of $\SO(3)$ with a single zero in some entry. 
\begin{theorem}
\label{thm:SO5}The variety $\SO^\star(5) \subseteq \mathbb{R}^{25}$ is the union of $64$ irreducible surfaces: $32$ tori of degree four and $32$ copies of $\SO^\circ(3) \times \SO(1) \times 
\SO(1)$ of degree eight. These $64$ components may be grouped into $16$ quadruples based on their zero pattern. There exists a bijection between these $16$ quadruples and the facets of the polytope $\mathcal P$ (illustrated on the right of \autoref{fig:polytopes}) satisfying the following:
\begin{enumerate}
\item \hspace{-2pt}The eight quadruples of tori correspond to the  quadrilaterals with two degree four vertices.
\item \hspace{-2pt}The eight quadruples of $\SO^\circ(3)\times \SO(1) \times 
\SO(1)$ correspond to the other facets.
\item \hspace{-2pt}Quadruples intersect in curves if and only if their corresponding facets share an edge. 
\item \hspace{-2pt}Quadruples intersect in isolated points if and only if their facets meet in a vertex. 
\item \hspace{-2pt}Curve intersections of components are copies of $\textrm{SO}(2) \times \SO(1) \times \SO(1) \times \SO(1) \cong \mathbb{S}^1$
\item \hspace{-2pt}Point intersections of components are signed permutation matrices of determinant one. 
\end{enumerate}
The relevant bijection is illustrated in \autoref{fig:SO5}. 
\end{theorem}

\subsection{The thirty-two tori and their intersections}
Each torus in $\SO^\star(5)$ appears as a $2+2+1$ embedding of $\SO(2) \times \SO(2) \times \SO(1)$ into  \eqref{eq:5zeropattern}. We identify the eight $2 \times 2$ blocks in \eqref{eq:5zeropattern}:
\[
\underbrace{\includegraphics[scale=0.36]{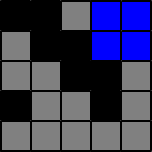}}_{A_1} \quad 
\underbrace{\includegraphics[scale=0.36]{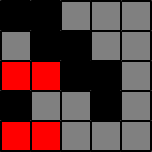}}_{A_2}\quad 
\underbrace{\includegraphics[scale=0.36]{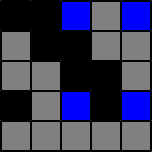}}_{A_3}\quad 
\underbrace{\includegraphics[scale=0.36]{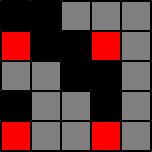}}_{A_4}\quad 
\underbrace{\includegraphics[scale=0.36]{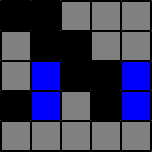}}_{A_5}\quad 
\underbrace{\includegraphics[scale=0.36]{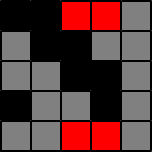}}_{A_6}\quad 
\underbrace{\includegraphics[scale=0.36]{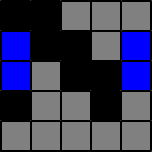}}_{A_7}\quad 
\underbrace{\includegraphics[scale=0.36]{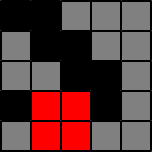}}
_{A_8}\]
To combine two of these $2 \times 2$ blocks and a $1 \times 1$ block into a $2+2+1$ block matrix one must superimpose two cyclically adjacent (modulo $8$) patterns to obtain a copy of $\O(2) \times \O(2) \times \O(1)$.  Fixing a sign choice for each $2 \times 2$ block, we obtain tori, like $T_{12+-}$ which is parametrized as 
\[ ((x_1,x_2),(y_1,y_2)) \in \mathbb{S}^1 \times \mathbb{S}^1 \mapsto
{\footnotesize{\begin{bmatrix}
0 & 0 & 0 & x_1 & -x_2 \\
0 & 0 & 0 & x_2 & x_1 \\
y_1 & y_2 & 0 & 0 & 0 \\
0 & 0 & 1 & 0 & 0 \\
y_2 & -y_1 & 0 & 0 & 0 
\end{bmatrix}}}.
\]  
There are eight possible block patterns, and four sign patterns for each, giving $32$ tori:
\[{{
\begin{array}{cccccccc}
T_{12++} & T_{23++} & T_{34++} & T_{45++} & T_{56++} & T_{67++} & T_{78++} & T_{81++} \\ 
T_{12+-} & T_{23+-} & T_{34+-} & T_{45+-} & T_{56+-} & T_{67+-} & T_{78+-} & T_{81+-} \\ 
T_{12-+} & T_{23-+} & T_{34-+} & T_{45-+} & T_{56-+} & T_{67-+} & T_{78-+} & T_{81-+} \\ 
T_{12--} & T_{23--} & T_{34--} & T_{45--} & T_{56--} & T_{67--} & T_{78--} & T_{81--} \\ 
\end{array}}}.
\]
\subsection{The thirty-two sections of $\SO(3)$ and their intersections}
Each $\O^\circ(3) \times \O(1) \times \O(1)$ pattern appearing in $\SO^\star(5)$ can be uniquely identified by the positions of each $\O(1)$ in \eqref{eq:5zeropattern} since the complement of the columns/rows indexed by two such entries of \eqref{eq:5zeropattern} determines a $3\times 3$ submatrix. For example, consider the pattern induced on \eqref{eq:5zeropattern} when $(1,4)$ and $(2,1)$ are $\pm 1$: 
\[
\includegraphics[scale=0.4]{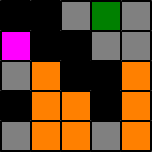}.
\] 
Only the $3 \times 3$ submatrices of \eqref{eq:5zeropattern} which have $8$ non-zero entries correspond to valid $3+1+1 = 5$ block patterns in \eqref{eq:5zeropattern}. As was the case for the tori in $\SO^\star(5)$, such a pattern, together with two signs, identify a component of $\SO^\star(5)$ of this form. We will denote such components like $C_{(1,4)(2,1)++}$ to indicate the indices of $\pm 1$'s and their signs. This gives $32$ components of degree $8 = \textrm{deg}(\SO(3))$:
\[{\footnotesize{
\begin{array}{cccc}
C_{(1,4)(2,1)++} & C_{(2,1)(3,2)++} & C_{(3,2)(4,3)++} & C_{(4,3)(1,4)++} \\ 
C_{(1,4)(2,1)+-} & C_{(2,1)(3,2)+-} & C_{(3,2)(4,3)+-} & C_{(4,3)(1,4)+-} \\ 
C_{(1,4)(2,1)-+} & C_{(2,1)(3,2)-+} & C_{(3,2)(4,3)-+} & C_{(4,3)(1,4)-+} \\ 
C_{(1,4)(2,1)--} & C_{(2,1)(3,2)--} & C_{(3,2)(4,3)--} & C_{(4,3)(1,4)--} \\ 
\end{array} 
\begin{array}{cccc}
C_{(1,3)(2,4)++} & C_{(2,4)(3,1)++} & C_{(3,1)(4,2)++} & C_{(4,2)(1,3)++} \\ 
C_{(1,3)(2,4)+-} & C_{(2,4)(3,1)+-} & C_{(3,1)(4,2)+-} & C_{(4,2)(1,3)+-} \\ 
C_{(1,3)(2,4)-+} & C_{(2,4)(3,1)-+} & C_{(3,1)(4,2)-+} & C_{(4,2)(1,3)-+} \\ 
C_{(1,3)(2,4)--} & C_{(2,4)(3,1)--} & C_{(3,1)(4,2)--} & C_{(4,2)(1,3)--} \\ 
\end{array}}}.
\]
We have now found all $64$ components of $\SO^\star(5)$ since
\[
32 \cdot 4 + 32 \cdot 8 = 384 =\textrm{deg}(\SO(5)).
\] The bijection of \autoref{thm:SO5} is given in \autoref{fig:SO5}.

\begin{figure}[!htpb]
 \includegraphics[scale=0.45]{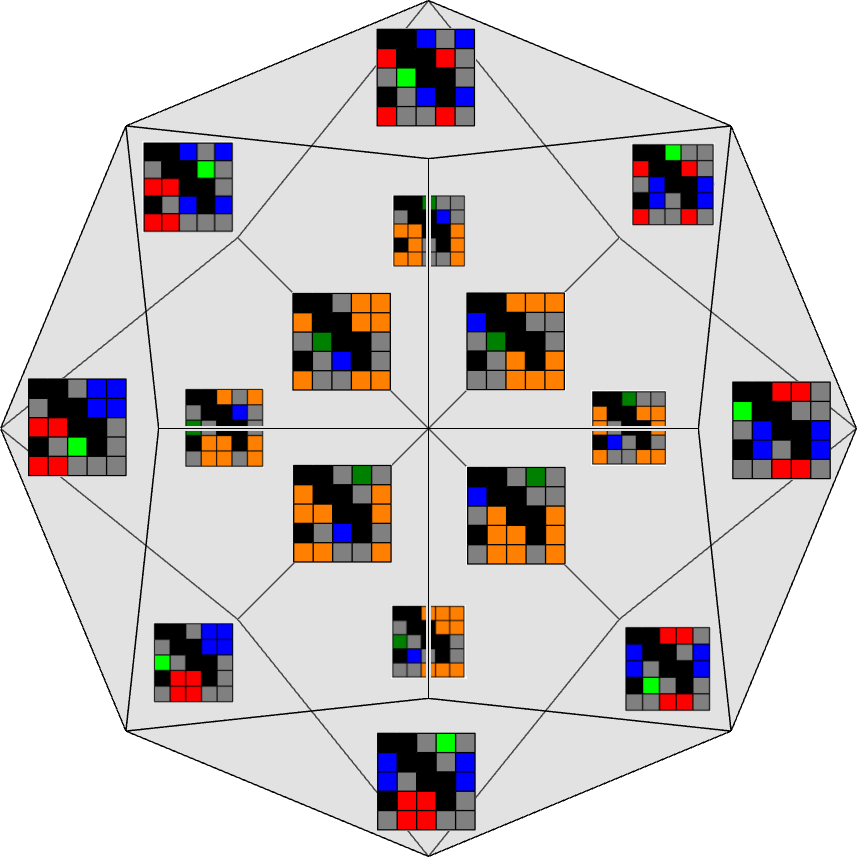}
 \caption{The polytope $\mathcal P$ with facets labeled by diagrams identifying quadruples of components of $\SO^\star(5)$ with the same zero pattern, giving the relevant bijection for \autoref{thm:SO5}.}
 \label{fig:SO5}
\end{figure}

\subsection{Two hundred fifty-six circles} 

The curve intersection incidence matrix of the $32$ tori is
\[
{\footnotesize{\begin{bmatrix}
0 & \mathcal I  & 0 & 0 & 0 & 0 & 0 & \mathcal I  \\ 
\mathcal I  & 0 & \mathcal I & 0 & 0 & 0 & 0 & 0   \\ 
0 & \mathcal I  & 0 & \mathcal I & 0 & 0 & 0 & 0    \\ 
0 & 0 & \mathcal I  & 0 & \mathcal I & 0 & 0 & 0    \\ 
0 & 0 & 0  & \mathcal I  & 0 & \mathcal I & 0 & 0    \\ 
0 & 0 & 0  & 0  & \mathcal I  & 0 & \mathcal I & 0    \\ 
0 & 0 & 0 & 0  & 0  & \mathcal I  & 0 & \mathcal I   \\ 
\mathcal I   & 0 & 0 & 0 & 0  & 0  & \mathcal I  & 0   \\ 
\end{bmatrix}} = \vcenter{\hbox{\includegraphics[scale=0.3]{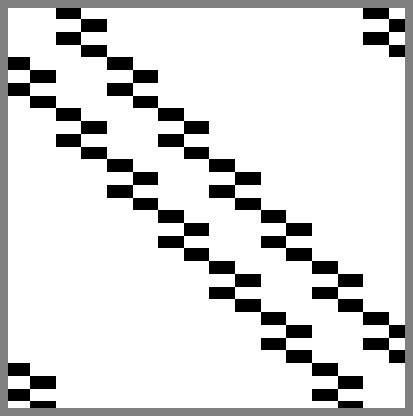}}}} \quad \quad \text{ where } \quad \quad \mathcal I = {\footnotesize{\begin{bmatrix} 1 & 1 & 0 & 0 \\
0 & 0 & 1 & 1 \\
1 & 1 & 0 & 0 \\
0 & 0 & 1 & 1 \end{bmatrix}}}.
\] Essentially, these components intersect exactly when they share a $2 \times 2$ nonzero block and there is a sign agreement. Thus, all intersections are copies of $\SO(2) \times \SO(1) \times \SO(1) \times \SO(1)$.

Two $\SO^\circ(3) \times \SO(1) \times \SO(1)$ components intersect in a curve if and only if they share a $1 \times 1$ block and the respective signs of that entry agree. The intersection incidence matrix is
\[
{\footnotesize{\begin{bmatrix}
0 & \mathcal I  & 0 &   \mathcal I & 0 & 0 & 0 & 0  \\ 
\mathcal I  & 0 & \mathcal I & 0 & 0 & 0 & 0 & 0   \\ 
0 & \mathcal I  & 0 & \mathcal I & 0 & 0 & 0 & 0    \\ 
\mathcal I & 0 & \mathcal I  & 0 & 0 & 0 & 0 & 0    \\ 
0 & 0 & 0  &0 & 0 & \mathcal I & 0 &  \mathcal I     \\ 
0 & 0 & 0  & 0  & \mathcal I  & 0 & \mathcal I & 0    \\ 
0 & 0 & 0 & 0  & 0  & \mathcal I  & 0 & \mathcal I   \\ 
0   & 0 & 0 & 0 &  \mathcal I   & 0  & \mathcal I  & 0   \\ 
\end{bmatrix}}} = \vcenter{\hbox{\includegraphics[scale=0.3]{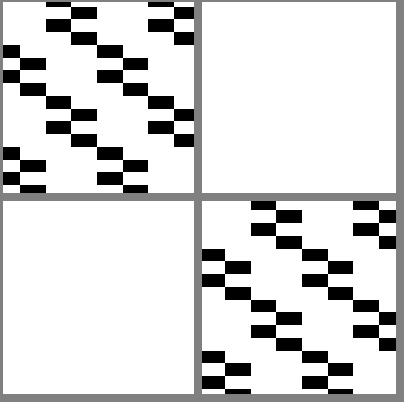}}}.
\]

To complete the incidence picture of the curve intersections of the $64$ components of $\textrm{SO}^*(5)$, we need only establish how the tori intersect the $\textrm{SO}^\circ(3) \times \textrm{SO}(1) \times \textrm{SO}(1)$ components. This can be done by hand using the combinatorics of the corresponding zero patterns, or via a direct computation in a computer algebra system, see \cite{BGmarvelousRepo}. For completeness, we provide the entire incidence matrix below in terms of $\mathcal I$ and the four matrices
\[
J_1 = \begin{bmatrix}
0 & 1 & 1 & 0 \\ 
1 & 0 & 0 & 1 \\
0 & 1 & 1 & 0 \\ 
1 & 0 & 0 & 1 
\end{bmatrix}, \quad \quad 
J_2 = \begin{bmatrix}
0 & 1 & 1 & 0 \\ 
0 & 1 & 1 & 0 \\ 
1 & 0 & 0 & 1 \\
1 & 0 & 0 & 1 
\end{bmatrix}, \quad \quad 
J_3 = \begin{bmatrix}
1 & 0 & 0 & 1 \\
0 & 1 & 1 & 0 \\ 
1 & 0 & 0 & 1 \\
0 & 1 & 1 & 0 
\end{bmatrix}, \quad \quad 
J_4 = \begin{bmatrix}
1 & 0 & 0 & 1 \\
1 & 0 & 0 & 1 \\
0 & 1 & 1 & 0 \\ 
0 & 1 & 1 & 0  
\end{bmatrix}.
\]
The complete curve intersection incidence matrix, along with its heatmap, is given below.
\begin{center}
${\footnotesize{\left[
\begin{array}{cccccccc||cccc|cccc}
0 & \mathcal I  & 0 & 0 & 0 & 0 & 0 & \mathcal I   & 0 & 0 & J_1 & J_2 & 0 & 0 & 0 & 0 \\
\mathcal I  & 0 & \mathcal I & 0 & 0 & 0 & 0 & 0   & 0 & 0 & 0 & 0 & J_3 & J_4 & 0 & 0 \\
0 & \mathcal I  & 0 & \mathcal I & 0 & 0 & 0 & 0   & 0 & J_1 & J_2 & 0 & 0 & 0 & 0 & 0 \\ 
0 & 0 & \mathcal I  & 0 & \mathcal I & 0 & 0 & 0   & 0 & 0 & 0 & 0 & J_4 & 0 & 0 & J_3 \\
0 & 0 & 0  & \mathcal I  & 0 & \mathcal I & 0 & 0  & J_1 & J_2 & 0 & 0 & 0 & 0 & 0 & 0 \\  
0 & 0 & 0  & 0  & \mathcal I  & 0 & \mathcal I & 0 & 0 & 0 & 0 & 0 & 0 & 0  & J_3 & J_4 \\   
0 & 0 & 0 & 0  & 0  & \mathcal I  & 0 & \mathcal I & J_2 & 0 & 0 & J_1 & 0 & 0 & 0 & 0 \\  
\mathcal I   & 0 & 0 & 0 & 0  & 0  & \mathcal I  & 0  &0 & 0 & 0 & 0 & 0 & J_3 & J_4 & 0 \\ \hline \hline 
0   & 0   & 0   & 0   & J_1 & 0   & J_2 & 0   &0 & \mathcal I  & 0 &   \mathcal I & 0 & 0 & 0 & 0  \\ 
0   & 0   & J_1 & 0   & J_2 & 0   & 0   & 0    &\mathcal I  & 0 & \mathcal I & 0 & 0 & 0 & 0 & 0   \\ 
J_1 & 0   & J_2 & 0   & 0   & 0   & 0   & 0    &0 & \mathcal I  & 0 & \mathcal I & 0 & 0 & 0 & 0    \\ 
J_2 & 0   & 0   & 0   & 0   & 0   & J_1 & 0    &\mathcal I & 0 & \mathcal I  & 0 & 0 & 0 & 0 & 0    \\  \hline 
0   & J_3 & 0   & J_4 & 0   & 0   & 0   & 0    &0 & 0 & 0  &0 & 0 & \mathcal I & 0 &  \mathcal I     \\ 
0   & J_4 & 0   & 0   & 0   & 0   & 0   & J_3    &0 & 0 & 0  & 0  & \mathcal I  & 0 & \mathcal I & 0    \\ 
 0  & 0   & 0   & 0   & 0   & J_3 & 0   & J_4   &0 & 0 & 0 & 0  & 0  & \mathcal I  & 0 & \mathcal I   \\ 
 0  & 0   & 0   & J_3 & 0   & J_4 & 0   & 0   &0   & 0 & 0 & 0 &  \mathcal I   & 0  & \mathcal I  & 0   
\end{array}\right] =}}$ $\vcenter{\hbox{\includegraphics[scale=0.3]{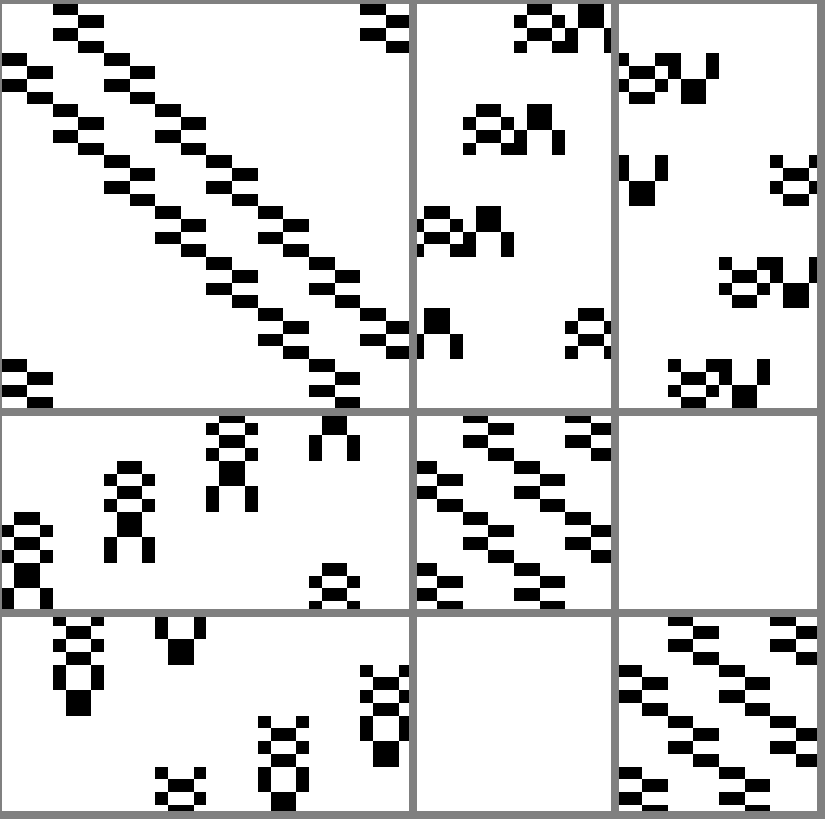}}}$
\end{center}
Each $\mathcal I, J_1,J_2,J_3,J_4$ involves eight ones. In total, this accounts for $512$ one's, representing  $256$ intersections. Each has the form of a circle realized as $\SO(2) \times \SO(1) \times \SO(1) \times \SO(1)$.  
\subsection{Two hundred eighty-eight signed permutation matrices}
There are $288$ unique zero-dimensional intersections of these $64$ components. Continuing the scheme of grouping elements by their zero patterns, these zero-dimensional intersections exhibit $18$ permutation matrix patterns; note that $288=2^4\cdot 18$. These $18$ permutation matrices are in bijection with the $18$ vertices of $\mathcal P$ in \autoref{fig:SO5Vertices} according to how they appear as intersections of the components. 
\begin{figure}[!htpb]
\includegraphics[scale=0.35]{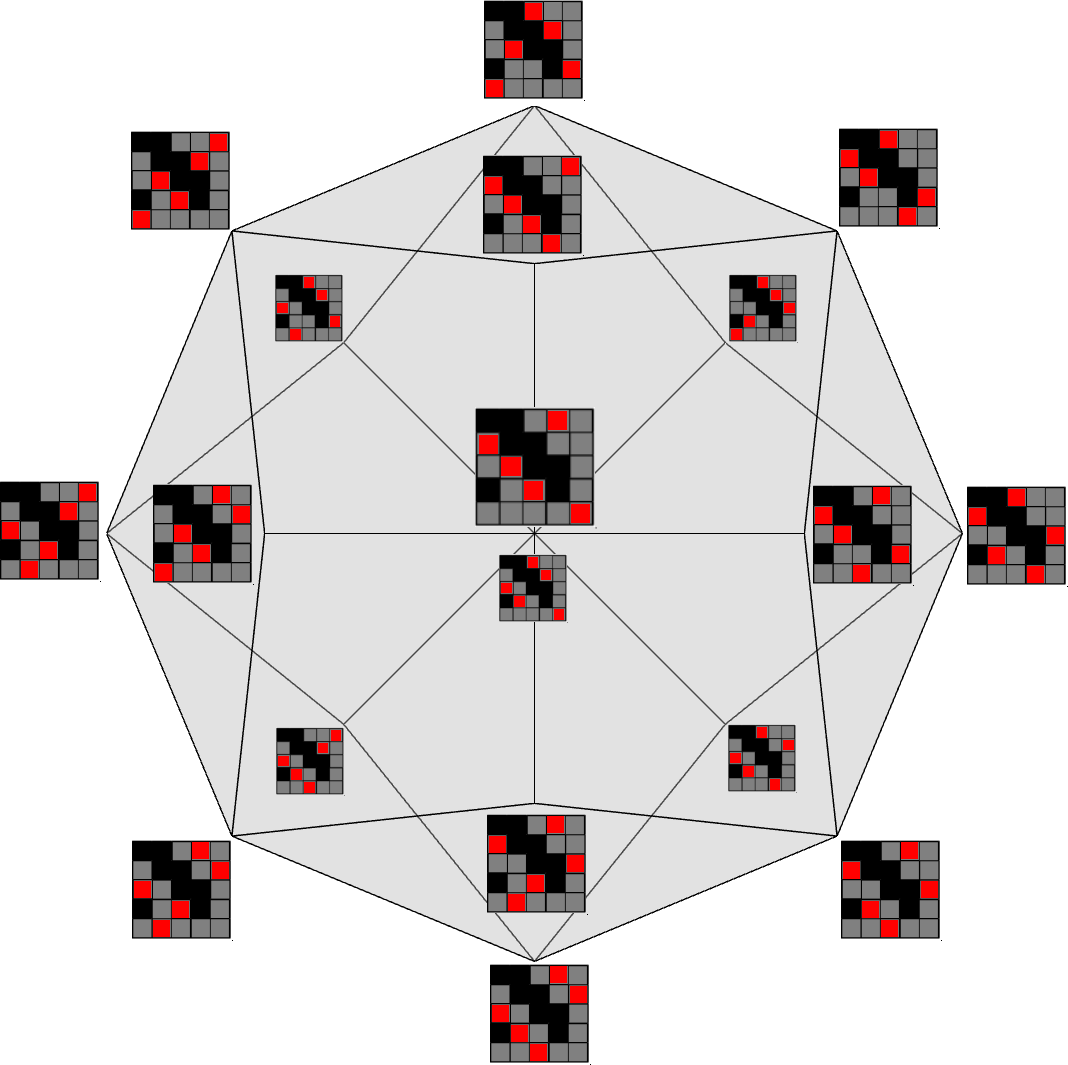}
\caption{The vertices of $\mathcal P$ correspond to the zero-dimensional intersections of components of $\SO^\star(5)$. Each represents $2^4$ signed permutation matrices.}
\label{fig:SO5Vertices}
\end{figure}

These $18$ vertices come in three flavours. These differences correspond to distinct cycle types.
\begin{enumerate}
\item Eight three-way intersections of components corresponding to the degree three vertices. Each is a $5$-cycle.
\item Two four-way intersections of $\SO^\circ(3) \times \SO(1) \times \SO(1)$ components. They have cycle types $(4,1)$ and $(2,2,1)$.
\item Eight intersections of one $\SO^\circ(3) \times \SO(1) \times \SO(1)$ and three torus components. Each has cycle type $(3,2)$.
\end{enumerate}
Each permutation $\sigma$ is subject to the condition that $\sigma(i) \not\in \{i,i+1 \textrm{ mod }4\}$ due to the zero pattern \eqref{eq:5zeropattern}. Unlike in $\HSO(4)$, these are \textit{all} of the permutation matrices which fit in the pattern \eqref{eq:5zeropattern}. 

\subsection{A totally real witness set of $\textrm{SO}(5)$} We obtained a totally real witness set for $\textrm{SO}(4)$ by choosing a special linear space through a subset of signed permutation matrices. In the case of $\SO(5)$, all of the signed permutation matrices compatible with the non-zero pattern \eqref{eq:5zeropattern} appear as intersections of components of $\SO^\star(5)$. In particular, each such point is singular and so a slice through any such permutation matrix would not be degree-generic. We leave the problem of using the $\SO^\star(5)$ decomposition to find a totally real witness set to future research.

\begin{conjecture}
There exists a codimension-two linear slice of $\SO^\star(5)$ which is totally real. 
\end{conjecture}

\section{Closing comments}
\subsection{A word on smaller orthogonal groups} We began with $n=4$ because the cases of $\SO(1)$ and $\SO(2)$ are trivial and $\SO(3)$ is almost trivial. The dimensions of $\SO(1)$ and $\SO(2)$ equal their rank. Thus, the empty zero pattern realizes each as a variety of dimension equal to its rank, in analogy with our previous results. The set of $3 \times 3$ special orthogonal matrices forms a three-dimensional variety of degree eight. Up to symmetries, there is a unique degree-generic coordinate slice of codimension two, namely, the one where two coordinates are set to zero which do not share a row or column. Such a slice $\{x_{11}=x_{22}=0\}$ decomposes $\textrm{SO}(3)$ into four circles:
\[
X_{12+}:
{\footnotesize{\begin{bmatrix}
0 & 1 & 0 \\
x_{21} & 0 & x_{23} \\
x_{31} & 0 & x_{33}
\end{bmatrix}}}, \quad \quad  X_{12-}:
{\footnotesize{\begin{bmatrix}
0 & -1 & 0 \\
x_{21} & 0 & x_{23} \\
x_{31} & 0 & x_{33}
\end{bmatrix}}}, \quad \quad  X_{21+}:
{\footnotesize{\begin{bmatrix}
0 & x_{12} & x_{13} \\
1 & 0 & 0 \\
0 & x_{31} & x_{33}
\end{bmatrix}}},\quad \quad  X_{21-}:
{\footnotesize{\begin{bmatrix}
0 & x_{12} & x_{13} \\
-1 & 0 & 0 \\
0 & x_{31} & x_{33}
\end{bmatrix}}}.
\]
The intersection incidence graph of these four curves is the polygon with four edges, with $X_{ij\pm}$ on opposite sides. The four vertices correspond to the transposition $(1,2)$ with signs.

\subsection{The necessity of the quadruples}
The beautiful story for $\textrm{SO}(4)$ struggles to extend to $\textrm{SO}(5)$: although there exists a unique zero pattern \eqref{eq:5zeropattern} which decomposes $\textrm{SO}(5)$ into many components of dimension $\textrm{rank}(\textrm{SO}(5)) = 2$, this pattern is more mysterious than simply being hollow, and one must combine components to encode the incidence structure in the face lattice of a polytope.

\begin{proposition}
There does not exist a $3$-polytope whose $64$ facets are in bijection with the components of $\SO^\star(5)$ such that its facets intersect in an edge if and only if the corresponding components intersect in a curve. 
\end{proposition}
\begin{proof}
The incidence matrix of the $64$ components has $512$ ones, and thus such a polytope would have $256$ edges. By Euler's formula, the number of vertices would have to be $194=256-64+2$. If every vertex had degree at least three then there would be at least $\frac{194\cdot 3}{2}=  291$ edges. So some vertex has degree less than three, a contradiction to it being a vertex of a $3$-polytope. 
\end{proof}

\subsection{Extending to $\textrm{SO}(6)$} 
For $n=6$, we have a result in the negative direction. Continuing the pattern for $n=4, 5$, we seek a zero pattern which is degree and dimension-generic and cuts $\SO(6)$ to  dimension three, that of its maximal torus. We hope such a pattern exists which decomposes $\textrm{SO}(6)$ significantly.  To find such a slice, one must identify a corresponding pattern of $12$ zeros.
\begin{theorem}
\label{thm:SO6}
There is no pattern of $12$ zeros in a $6 \times 6$ matrix which represents a slice of codimension $12$ that is degree and dimension-generic with respect to $\textrm{SO}(6)$. 
\end{theorem}
\begin{proof}
This is proven by exhaustive computation over all $(6,6)$-bipartite graphs with $12$ edges and no vertex of degree six  or higher. The computation first collects all $5816$ such graphs up to symmetry using Nauty \cite{Nauty}. For each, we symbolically compute the dimension and degree of the corresponding slice in Macaulay2 \cite{Macaulay2}; no dimension-generic and no degree-generic section arises (even over complex numbers). The full results are available at \cite{BGmarvelousRepo}.
\end{proof}
\autoref{thm:SO6} tells us that generalizing what is done in this paper to $n=6$ is impossible if one hopes for too many of the properties of the previous cases to persist. 
 An alternative approach is to interpret the zero pattern as affine spans of signed permutation matrices. Generalizing from this point of view requires deep knowledge of the \textit{algebraic matroid} of signed permutation matrices. 

\subsection{The polytope $\mathcal P$} As far as the authors are aware, the polytope $\mathcal P$ illustrated on the right of \autoref{fig:polytopes} has never been identified in a research article. Like the Johnson solid J17, our polytope $\mathcal P$ is comprised of two halves: the near and far halves of the right of \autoref{fig:polytopes}. In both J17 and $\mathcal P$, these halves are twisted by $45$ degrees. However, $\mathcal P$ does not have an edge-cycle in a hyperplane equator (like a bicupola) nor a belt (like an elongated bicupola). Our polytope has the same $f$-vector as J28 and J29. Moreover, the combinatorial symmetries of J29 and $\mathcal P$ are both $D_{16}$.  Nevertheless, our polytope $\mathcal P$ is combinatorially distinct from all of these examples.

{ \small 
\bibliographystyle{alphaurl}
\bibliography{references}

\newcommand{\etalchar}[1]{$^{#1}$}
\begin{thebibliography}{BBB{\etalchar{+}}17}

\bibitem[BBB{\etalchar{+}}17]{SOn}
M.~Brandt, J.~Bruce, T.~Brysiewicz, R.~Krone, and E.~Robeva.
\newblock The degree of {SO(n,$\mathbb{C}$)}.
\newblock {\em Combinatorial Algebraic Geometry}, pages 229--246, 2017.
\newblock \href {https://doi.org/10.1007/978-1-4939-7486-3_11}
  {\path{doi:10.1007/978-1-4939-7486-3_11}}.

\bibitem[BG21]{StiefelManifolds}
T.~Brysiewicz and F.~Gesmundo.
\newblock {The Degree of Stiefel Manifolds}.
\newblock {\em Enum. Comb. Appl.}, 1(3):n.S2R20, 2021.
\newblock \href {https://doi.org/10.54550/ECA2021V1S3R20}
  {\path{doi:10.54550/ECA2021V1S3R20}}.

\bibitem[BG26]{BGmarvelousRepo}
T.~Brysiewicz and F.~Gesmundo.
\newblock Marvelous slices of orthogonal matrices.
\newblock
  \url{https://github.com/fulges/Marvelous_Slices_of_Orthogonal_Matrices},
  2026.
\newblock GitHub repository accompanying the paper.

\bibitem[Col10]{Colbourn2010}
C.~J. Colbourn.
\newblock {\em CRC handbook of combinatorial designs}.
\newblock CRC press, 2010.

\bibitem[DEF{\etalchar{+}}24]{OscarBook}
W.~Decker, C.~Eder, C.~Fieker, M.~Horn, and M.~Joswig, editors.
\newblock {\em The {C}omputer {A}lgebra {S}ystem {OSCAR}: {A}lgorithms and
  {E}xamples}, volume~32 of {\em Algorithms and {C}omputation in
  {M}athematics}.
\newblock Springer, 1 edition, 8 2024.
\newblock URL: \url{https://link.springer.com/book/9783031621260}.

\bibitem[EH16]{3264}
D.~Eisenbud and J.~Harris.
\newblock {\em 3264 \& All That: A Second Course in Algebraic Geometry}.
\newblock Cambridge University Press, 2016.

\bibitem[Gen07]{gentle2007matrix}
J.~E. Gentle.
\newblock {\em {Matrix algebra: theory, computations, and applications in
  statistics}}.
\newblock Springer, 2007.

\bibitem[GJ00]{Polymake}
E.~Gawrilow and M.~Joswig.
\newblock polymake: a framework for analyzing convex polytopes.
\newblock In {\em Polytopes---Combinatorics and Computation (Oberwolfach,
  1997)}, volume~29 of {\em DMV Sem.}, pages 43--73. Birkh\"auser, Basel, 2000.

\bibitem[GS26]{Macaulay2}
D.~R. Grayson and M.~E. Stillman.
\newblock Macaulay2, a software system for research in algebraic geometry.
  v.1.21.
\newblock Available at \url{http://www2.macaulay2.com}, 2026.

\bibitem[Har92]{HarrisAG}
J.~Harris.
\newblock {\em Algebraic Geometry: A First Course}, volume 133 of {\em Graduate
  Texts in Mathematics}.
\newblock Springer, New York, 1992.

\bibitem[Jou83]{BertiniBook}
J.-P. Jouanolou.
\newblock {\em Th{\'e}or{\`e}mes de Bertini et applications}, volume~42 of {\em
  Progress in Mathematics}.
\newblock Birkh{\"a}user, Boston, 1983.

\bibitem[Kaz87]{Kazarnovskii}
B.~Y. Kazarnovskii.
\newblock {Newton polyhedra and the Bezout formula for matrix-valued functions
  of finite-dimensional representations}.
\newblock {\em Func. An. Appl.}, 21(4):319–321, 1987.
\newblock \href {https://doi.org/10.1007/BF01077809}
  {\path{doi:10.1007/BF01077809}}.

\bibitem[MP13]{Nauty}
B.~D. McKay and A.~Piperno.
\newblock Practical graph isomorphism, ii.
\newblock {\em Journal of Symbolic Computation}, 60:94--112, 2013.
\newblock \href {https://doi.org/10.1016/j.jsc.2013.09.003}
  {\path{doi:10.1016/j.jsc.2013.09.003}}.

\bibitem[Nic25]{Nic}
D.~R. Nicholus.
\newblock {Zeroing Diagonals, Conjugate Hollowization, and Characterizing
  Nondefinite Operators}.
\newblock {\em arXiv:2508.00096}, 2025.

\bibitem[NRS10]{LowRankSDP}
J.~Nie, K.~Ranestad, and B.~Sturmfels.
\newblock The algebraic degree of semidefinite programming.
\newblock {\em Mathematical Programming, Series A}, 122(2):379--405, 2010.
\newblock \href {https://doi.org/10.1007/s10107-008-0253-6}
  {\path{doi:10.1007/s10107-008-0253-6}}.

\bibitem[OSC24]{Oscar}
Oscar -- open source computer algebra research system, version 1.2.0-dev, 2024.
\newblock URL: \url{https://www.oscar-system.org}.

\bibitem[SVW05]{NAG}
A.~J. Sommese, J.~Verschelde, and C.~W. Wampler.
\newblock Introduction to numerical algebraic geometry.
\newblock In {\em Solving polynomial equations}, volume~14 of {\em Algorithms
  Comput. Math.}, pages 301--335. Springer, Berlin, 2005.
\newblock \href {https://doi.org/10.1007/3-540-27357-3_8}
  {\path{doi:10.1007/3-540-27357-3_8}}.

\bibitem[Tro02]{trosset2002extensions}
M.~W. Trosset.
\newblock {Extensions of classical multidimensional scaling via variable
  reduction}.
\newblock {\em Computational Statistics}, 17(2):147--163, 2002.
\newblock \href {https://doi.org/10.1007/s001800200099}
  {\path{doi:10.1007/s001800200099}}.

\end{thebibliography}
}
\end{document}